\documentclass[12pt]{article}
\usepackage{natbib}

\usepackage{amsmath,amssymb,amsthm,amsfonts}
\usepackage{graphicx,url}
\usepackage{latexsym}

\usepackage{moreverb}
\usepackage{url}
\usepackage{grffile}
\input{undertilde}

\newcommand{\E}{\mbox{$\rm E$}}

\newcommand{\qq}{q}

\newcommand{\hf}{\textstyle{\frac{1}{2}}}

\newcommand{\Ai}{\mbox{\rm Ai}}

\newcommand\BibTeX{{\rmfamily B\kern-.05em \textsc{i\kern-.025em b}\kern-.08em
T\kern-.1667em\lower.7ex\hbox{E}\kern-.125emX}}
\newcommand\MiKTeX{{\rmfamily M\kern-.05em \textsc{i\kern-.025em K}\kern-.08em
T\kern-.1667em\lower.7ex\hbox{E}\kern-.125emX}}
\newcommand\PracTeX{{\rmfamily P\kern-.05em \textsc{r\kern-.025em a\kern-.025em
c}\kern-.08em
T\kern-.1667em\lower.7ex\hbox{E}\kern-.125emX}}

\theoremstyle{plain}
\newtheorem{thm}{Theorem}

\begin{document}

\title{Tail sums of Wishart and GUE eigenvalues beyond the bulk
  edge.}
\author{Iain M. Johnstone\thanks{Thanks
to David Donoho and Matan Gavish for discussions and help.
Peter Forrester provided an important reference. 
Work on this paper was completed during a visit to the Research
School of Finance, Actuarial Studies and Statistics, A. N. U.,
whose hospitality 
and support is gratefully acknowledged.
Work supported in part by the U.S. National Science
Foundation and the National Institutes of Health.} \\
Stanford University and Australian National University
}
\date{\today}

\maketitle

\begin{abstract}
Consider the classical Gaussian unitary ensemble of size $N$ and
the real Wishart ensemble $W_N(n,I)$.
In the limits as $N \to \infty$
and $N/n \to \gamma > 0$, 
the expected number of eigenvalues that
exit the upper bulk edge is less than one, 0.031 and 0.170
respectively, the latter number being independent of $\gamma$.
These statements are consequences of quantitative bounds on tail sums
of eigenvalues outside the bulk which are
established here for applications in high 
dimensional covariance matrix estimation. \\
\end{abstract}


\section{Introduction}
\label{sec:introduction}

This paper develops some tail sum bounds on eigenvalues
outside the bulk that are needed for results on estimation of
covariance matrices in the spiked model,
\citet{dgj15}.
This application is described briefly in Section
\ref{sec:appl-covar-estim}. It depends on properties of the
eigenvalues of real \textit{white} Wishart matrices,
distributed as $W_N(n, I)$, which
are the main focus of this note.

Specifically, suppose that $A \sim W_N(n,I)$, and 
that $\lambda_1 \geq \ldots \geq \lambda_N$ are eigenvalues of the sample
covariance matrix $n^{-1} A$.
In the limit $ N/n \to \gamma > 0$, it is well known that the
empirical distribution of $\{ \lambda_i \}$ converges to 
the Marcenko-Pastur law
(see e.g. \citet[Corollary 7.2.5]{pash11}), which is supported on an
interval $I_\gamma$ 
--- augmented with $0$ if $\gamma > 1$ ---
having upper endpoint $\lambda(\gamma) = (1 + \sqrt{\gamma})^2$. We focus
on the eigenvalues $\lambda_i$ that exit this ``bulk''
interval $I_\gamma$ on the upper side. 
In statistical application, 
such exiting eigenvalues might
be mistaken for ``signal'' and so it is useful to have some bounds on
what can happen under the null hypothesis of no signal.
Section \ref{sec:real-wishart-case} studies
the mean value behavior of  quantities such as
\begin{equation*}
  T_N = \sum_{i=1}^N [\lambda_i - \lambda(\gamma)]_+^q, \qquad q \geq 0
\end{equation*}
which for $q = 0$ reduces to the number $T_N^0$ of exiting eigenvalues.

It is well known that the largest eigenvalue $\lambda_1 \stackrel{\rm 
  a.s.}{\to} \lambda(\gamma)$ \citep{gema80},
and that closed intervals outside the bulk support contain no
eigenvalues for $N$ large with probability one \citep{basi98}.
However these and even large deviation results for 
$\lambda_1$ \citep{mave09}
and $T_N^0$ \citep{mavi12}
seem not to directly yield the information on $\E (T_N)$ that we need.
\citet{mmsv14} looked at the \textit{variance} of $T_N^0$ using
methods related to those of this note.
Recently, \citet{chia17} has studied the probability that all eigenvalues
of Gaussian, Wishart and double Wishart random matrices lie within the bulk, and
derived universal limiting values of 0.6921 and 0.9397 in the real and complex
cases respectively.

In summary, the motivation for this note is high-dimensional
covariance estimation, but there are noteworthy byproducts:
the asymptotic values of $\E (T_N^0)$ are perhaps suprisingly small, and
numerically for the Gaussian Unitary Ensemble (GUE), 
it is found that the chance of even two exiting
eigenvalues is \textit{very} small, of order $10^{-6}$.

\section{The Gaussian Unitary Ensemble Case (GUE)}

We begin with GUE  to illustrate the methods in the simplest
setting, and to note an error in the literature. 
Recall that the Gaussian Unitary ensemble GUE($N$) is the Gaussian probability
measure on the space of $N \times N$ Hermitian matrices with density
proportional to $\exp \{ - \hf N \text{tr} A^2 \}$.

\begin{thm}
Let $\lambda_1,\ldots,\lambda_N$ be eigenvalues of an $N$-by-$N$ matrix from the
GUE. Denote by $\lambda_+=2$ the upper edge of the Wigner semicircle, namely,
the asymptotic density of the eigenvalues. For $ \qq \geq 0$, let
\begin{eqnarray}
    T_N =  \sum_{i=1}^N (\lambda_i-\lambda_+)_+^{\qq}\,.
\end{eqnarray}
Then, with a constant $c_\qq$ specified at \eqref{eq:cqdef} below,
\[
\E (T_N) = c_\qq N^{-2\qq/3}(1 + o(1)).
\]
In particular, for $\qq = 0$ and $T_N = \# \{ i ~:~ \lambda_i >
\lambda_+ \}$,
\begin{equation}  \label{eq:result}
  \E (T_N) \to c_0 = \frac{1}{6 \sqrt{3} \pi} \approx 0.030629.
\end{equation}

\end{thm}

\begin{proof}
    We use the so-called {\em one-point function} and bounds due
to \cite{trwi94,trwi96}.
To adapt to their
notation, let $(y_i)_1^N$ be the eigenvalues of GUE with joint density
proportional to $\exp (-\sum_1^n y_i^2 ) \Delta^2(y)$, where
$\Delta(y)$ is the usual Vandermonde. In this scaling the eigenvalue
bulk concentrates as the semi-circle on $[-\mu_N, \mu_N]$ with $\mu_N
= \sqrt{2N}$. 

We have $y_i = \sqrt{N/2} \, \lambda_i$ and $\mu_N = \sqrt{N/2} \,
\lambda_+$, for $\lambda_+ = 2$, so that
\begin{equation*}
  T_N = \sum_1^N (\lambda_i - \lambda_+)_+^{\qq}
      = \Big( \frac{2}{N} \Big)^{\qq/2} \sum_1^N (y_i-\mu_N)_+^{ \qq}.
\end{equation*}

From the determinantal structure of GUE, the marginal density of a
single (unordered) eigenvalue $y_i$ is given by the one-point function
\begin{equation*}
  N^{-1} S_N(y,y) = N^{-1} \sum_{k=0}^{N-1} \phi_k^2(y),
\end{equation*}
where $\phi_k(y)$ are the (Hermite) functions obtained by
orthonormalizing $y^k e^{-y^2/2}$. Thus
\begin{equation*}
  \E (T_N) = \Big( \frac{2}{N} \Big)^{\qq/2} \int_{\mu_N}^\infty
  (y-\mu_N)^{\qq} S_N(y,y) dy.
\end{equation*}
Now introduce the TW scaling
\begin{equation*}
  y = \mu_N + \tau_N x, \qquad \qquad \tau_N = \frac{1}{\sqrt{2} N^{1/6}},
\end{equation*}
and let $\text{Ai}$ denote the Airy function.
\citet[p 745-6]{trwi96} show that
\begin{align*}
  S_{\tau_N}(x,x) 
   & = \tau_N S_N( \mu_N + \tau_N x, \mu_N + \tau_N x) \\
   & \to K_A(x,x) = \int_0^\infty \text{Ai}^2(x+z) dz,
\end{align*}
with the convergence being dominated: $S_{\tau_N}(x,x) \leq M^2
e^{-2x}$.
Consequently,
\begin{align*}
  \E (T_N) 
   & = \Big( \frac{2 \tau_N^2}{N} \Big)^{\qq/2} \int_0^\infty x^{\qq}
   S_{\tau_N}(x,x) dx \\
   & \sim N^{-2 \qq/3} \int_0^\infty x^{\qq} K_A(x,x) dx.
\end{align*}
In particular, $\E (T_N) = O(N^{-2 \qq/3})$, and if $\qq = 0$, then
$\E (T_N)$ converges to a positive constant.

Integration by parts and 
\citet[9.11.15]{Olver:2010:NHMF}  yield
\begin{equation}
  \label{eq:cqdef}
  \begin{split}
  c_\qq & =  \int_0^\infty x^{\qq} K_A(x,x) dx  
      = \int_0^\infty x^\qq \int_x^\infty \Ai^2(z) dz dx  \\
      &  = \frac{1}{\qq+1} \int_0^\infty x^{\qq + 1} \Ai^2(x) dx 
      =  \frac{ 2 \Gamma(\qq + 1)}{\sqrt{\pi} 12^{(2\qq +
          9)/6} \Gamma((2\qq+9)/6)}
  \end{split}
\end{equation}
For $\qq=0$ the constant becomes $c_0 = 1/(6 \sqrt{3} \pi)$. 
\end{proof}

\begin{table}[h]
  \centering
  \caption{For GUE$(N)$, the probabilities $p_N(k)$ of exactly $k$ eigenvalues
  exceeding the upper bulk edge $\sqrt{2N}$, along with the expected
  number $\E (T_N)$, to be compared with limiting value \eqref{eq:result}.}
\begin{tabular}[h]{ccccc}
\quad  $N$ \quad \qquad  & \qquad  $p_N(1)$ \qquad \ & \qquad  $p_N(2)$ \qquad \ & \qquad
  $p_N(3)$ \qquad \  &\qquad  $\E (T_N)$ \\[8pt] 
10  &$2.868\cdot10^{-2} $& $1.36\cdot10^{-6}$& $6.9\cdot10^{-14}$&   0.028681 \\
25  &$2.955\cdot10^{-2} $& $1.70\cdot10^{-6}$& $1.4\cdot10^{-13}$&   0.029551 \\
50  &$2.994\cdot10^{-2} $& $1.88\cdot10^{-6}$& $1.9\cdot10^{-13}$&   0.029944 \\
100 &$3.019\cdot10^{-2} $& $2.00\cdot10^{-6}$& $2.3\cdot10^{-13}$&   0.030195 \\
250 &$3.039\cdot10^{-2} $& $2.09\cdot10^{-6}$& $2.6\cdot10^{-13}$&   0.030392 \\
500 &$3.048\cdot10^{-2} $& $2.14\cdot10^{-6}$& $2.8\cdot10^{-13}$&   0.030480
\end{tabular}
  \label{tab:born}
\end{table}

\bigskip
\textbf{Remarks.} \ 
1. \citet{ulla83} states, in our notation, that the expected number of
eigenvalues above the bulk edge, $\E (T_N) \sim 0.25 N^{-1/2}$. 
This claim cannot be correct: a  counterexample uses the
limiting law $F_2$ for $y_{(1)} = \max_i y_i$ of \citet{trwi94}:
\begin{equation}
  \label{eq:TW}
  \E (T_N) \geq
  \Pr ( y_{(1)} > \sqrt{2N} ) 
  \to 1 - F_2(0) = 0.030627.
\end{equation}

We evaluated 
numerically in \textsf{Mathematica}
the  formulas (U3), (U6) and (U7) for $p = (2/N) \E (T_N)$ given in 
\citet{ulla83}.
While numerical results from intermediate formula (U3) are consistent with our 
(\ref{eq:result}), neither those from (U6) nor those from the final
result (U7) are consistent with (U3), or indeed with each other!

2. The striking closeness of the right side of (\ref{eq:TW}) to 
(\ref{eq:result}) led us to use the \textsf{Matlab} toolbox of
\citet{born10} to evaluate numerically
\begin{equation*}
  p_N(k) 
  = \Pr ( \text{ exactly } k \text{ of } \{y_i\} > \sqrt{2N} ) 
  = E_2^{(n)}(k, J)
\end{equation*}
with $J = (\sqrt{2N},\infty)$, in the notation of \citet{born10}.
The results, in Table \ref{tab:born}, confirm that the probability of
2 or more eigenvalues exiting the bulk is very small, of order
$10^{-6}$, for all $N$. 
This is also suggested by the plots of the densities of $y_{(1)},
y_{(2)}, \ldots$ in the scaling limit in Figure 4 of \citet{born10},
which itself extends Figure 2 of \citet{trwi94}.

\section{The real Wishart case}
\label{sec:real-wishart-case}

Suppose $\lambda_i$ are eigenvalues of $n^{-1} X X^\top$ for $X$ a $N
\times n$ matrix with i.i.d. ${\rm N}(0,1)$ entries. Assume that
$\gamma_N = N/n \to \gamma \in (0,1]$.
Set $\lambda(\gamma) = (1 + \sqrt \gamma)^2$. 

We recall the scaling for the Tracy-Widom law from the largest
eigenvalue $\lambda_1$:
\begin{equation*}
  \lambda_1 = \lambda(\gamma_N) + N^{-2/3} \tau(\gamma_N) W_N
\end{equation*}
where $W_N$ converges in distribution to $W \sim TW_1$ and 
$\tau(\gamma) = \sqrt \gamma ( \sqrt \gamma + 1)^{4/3}.$

\bigskip

\begin{thm}
  \label{prop:wish-case}
(a) Suppose $\eta(\lambda,c) \geq 0$ is jointly continuous in
$\lambda$ and $c$, and satisfies
\begin{alignat*}{2}
  \eta(\lambda, c) \
    & = \ 1 \qquad \quad && \text{for } \lambda \leq \lambda(c) \\
  \eta(\lambda, c) \ 
    &  \leq \ M \lambda && \text{for some } M \ \text{and all } \lambda.
\end{alignat*}
Suppose also that $c_N - \gamma_N = O(N^{-2/3})$. Then for
$q > 0$, 
\begin{equation}
  \label{eq:a-cge}
  \E \bigg( \sum_{i =1}^N [\eta(\lambda_i, c_N) - 1]^q \bigg) \to 0.
\end{equation}
(b) Suppose $c_N - \gamma_N \sim s \sigma(\gamma) N^{-2/3}$, where 
$\sigma(\gamma) = \tau(\gamma)/\lambda'(\gamma)
= \gamma (1 + \sqrt \gamma)^{1/3}$.
Then 
\begin{equation} \label{eq:cge-b}
  \E \bigg( \sum_{i =1}^N [\lambda_i - \lambda(c_N)]_+^q \bigg) 
      \sim \tau^q(\gamma) N^{-2q/3} \int_s^\infty (x-s)_+^q K_1(x,x) dx.
\end{equation}
where $K_1$ is defined at \eqref{eq:K1} below.

\noindent(c) In particular, let $N_n = \# \{i ~:~ \lambda_i \geq \lambda(c_N)
\}$ and suppose that $c_N - \gamma_N = o(N^{-2/3})$. Then
\begin{equation*}
   \E N_n \to c_0
          = \int_0^\infty K_1(x,x) dx \approx 0.17.
\end{equation*}
\end{thm}

\textit{Remarks.}
1. Part (b) represents a sharpening of \eqref{eq:a-cge} that is relevant
when $\eta(\lambda) = \eta(\lambda,\gamma)$ is H\"older continuous in
$\lambda$ near the bulk edge $\lambda(\gamma)$,
$$ \eta(\lambda) - \eta(\lambda(\gamma)) \sim (\lambda -
\lambda(\gamma))_+^\qq.$$
The example $\qq = 1/2$ occurs commonly for optimal shrinkage rules
$\eta^*(\lambda)$ in \citet{dgj15}.

2. Section \ref{sec:appl-covar-estim} explains
 why we allow $c_N$ to differ from $\gamma_N$.

\begin{proof}
  Define
\begin{equation*}
  T_N = \sum_{i=1}^N F(\lambda_i,c_N), \qquad \qquad 
  F(\lambda,c) =
  \begin{cases}
    [\eta(\lambda,c) -  1]^q  &  (a) \\
    [\lambda - \lambda(c)]_+^q         &  (b).
  \end{cases}
\end{equation*}
We adapt the discussion here
 to the notation used in \citet{trwi98} and \citet{john00c}.
Let $(y_i)_1^N = n \lambda_i$ be the eigenvalues of $W_N(n,I)$ with
joint density 
function $P_N(y_1, \ldots, y_N)$ with explicit form given, for
example, in \cite[eq. (4.1)]{john00c}.
We obtain
\begin{equation*}
  \E (T_N) = \int_0^\infty F(y/n, c_N) R_1(y) dy, 
\end{equation*}
where $R_1(y_1) = N \int_{(0,\infty)^{N-1}} P_N(y_1, \ldots, y_N) dy_2
\cdots dy_N$ is the one-point (correlation) function. 
It follows from \citet[p814--16]{trwi98} that
\begin{equation}
    \label{eq:R1}
  R_1(y) = T_1(y) = \hf \text{tr } K_N(y,y)
\end{equation}
where $K_N(x,y)$ is the $2 \times 2$ matrix kernel associated with
$P_N$, see e.g. \cite[eq. (3.1)]{trwi98}. 
It follows from \citet{wido99} that
\begin{equation} \label{eq:S1}
         \hf \text{tr } K_N(y,y) 
       = S(y,y) + \psi(y) (\epsilon \phi)(y)
       =  S_1(y,y),
\end{equation}
where the functions $S(y,y'), \psi(y)$ and $\phi(y)$  are defined in terms of orthonormalized Laguerre
polynomials in \cite{wido99} and studied further in \cite{john00c}.
The function $\epsilon(x) = \hf \text{sgn} x$ and the operator
$\epsilon$ denotes convolution with the kernel $\epsilon(x-y)$.
\end{proof}
For convergence, introduce the Tracy-Widom scaling
\begin{equation*}
  y = \mu_N + \sigma_N x,
\end{equation*}
where we set $N_h = N + \hf$ and  $n_h = n + \hf$ and define
\begin{equation*}
  \mu_N = (\sqrt{N_h} + \sqrt{n_h})^2, \qquad
  \sigma_N = c(N_h/n_h) N_h^{1/3},
\end{equation*}
where $c(\gamma) = (1 + \sqrt \gamma)^{1/3} (1 + 1/\sqrt \gamma) 
=  (1 + \sqrt \gamma)^{1/3} \lambda'(\gamma)$
We now rescale the scalar-valued function \eqref{eq:S1}:
\begin{equation*}
  S_{1 \tau} (x, x)
     = \sigma_N S_1(\mu_N + \sigma_N x, \mu_N + \sigma_N x).
\end{equation*}
We can rewrite our target $\E (T_N)$ using \eqref{eq:R1}, \eqref{eq:S1}
and this rescaling in the form
\begin{equation*}
  \E (T_N) =  \int_{\delta_N}^\infty 
      F(\ell_N (x), c_N)
                 S_{1 \tau}(x,x) dx,
\end{equation*}
where $\ell_N(x) = (\mu_N + \sigma_N x)/n$, 
$\delta_N = (n \lambda(c_N) - \mu_N)/\sigma_N$ and we used the fact that
$F(\lambda,c) = 0$ for $\lambda \leq \lambda(c)$. 

It follows from 
\cite[eq. (3.9)]{john00c} that
\begin{equation*}
  S_{1 \tau} (x, x)
    = 2 \int_0^\infty \phi_\tau(x+u) \psi_\tau(x+u) du
        + \psi_\tau(x) \left[ c_\phi - \int_x^\infty \phi_\tau(u) du
        \right]. 
\end{equation*}
It is shown in equations (3.7), 3.8) and Sec. 5 of that paper that
\begin{equation*}
  \phi_\tau(x), \psi_\tau(x) \to \frac{1}{\sqrt 2} \Ai(x)
\end{equation*}
and, uniformly in $N$ and in intervals of $x$ that are bounded below, that
\begin{equation*}
  \phi_\tau(x), \psi_\tau(x) = O(e^{-x}).
\end{equation*}
Along with $c_\phi \to 1/\sqrt{2}$ (cf. App. A7 of same paper),
this shows that
\begin{equation} \label{eq:K1}
  S_{1 \tau}(x, x) 
    \to K_1(x, x)
    = \int_0^\infty \text{Ai}^2(x+z) dz
 + \tfrac{1}{2} \Ai(x) \left[ 1 - \int_x^\infty \Ai(z) dz \right] > 0
\end{equation}
with the convergence being dominated
\begin{equation}
\label{eq:S-dom}
  S_{1 \tau}(x, x) \leq M^2 e^{-2x} + M' e^{-x}.
\end{equation}

Before completing the argument for  (a) -- (c), we note it is
easily checked that
\begin{equation}
  \label{eq:muNbd}
  n^{-1} \mu_N = \lambda(\gamma_N) + O(N^{-1}),
\end{equation}
so that
\begin{equation*}
  \delta_N = \frac{n}{\sigma_N}[ \lambda(c_N) - \lambda(\gamma_N)] +
  O(N^{-1/3}). 
\end{equation*}
If $c_N - \gamma_N = \theta_N N^{-2/3}$ for $\theta_N = O(1)$ then
\begin{equation*}
  \delta_N \sim \frac{n}{\sigma_N} N^{-2/3} \theta_N \lambda'(\gamma) 
   \sim \theta_N/\sigma(\gamma),
\end{equation*}
since we have
\begin{equation}
  \label{eq:asyrel}
  N^{2/3} \sigma_N/n \sim \sigma(\gamma) \lambda'(\gamma) = \tau(\gamma).
\end{equation}

In case (a), then, 
$\delta_N \geq - A$ for some $A$. 
We then have $\ell_N(x) \to \lambda(\gamma)$ for all $x \geq -A$, and so 
from joint continuity
\begin{equation*}
  \eta( \ell_N(x), c_N) \to \eta(\lambda(\gamma),
  \gamma) = 1, 
\end{equation*}
and hence for all $x \geq -A$,
\begin{equation}
  \label{eq:Fto0}
  F( \ell_N (x), c_N) 
  = [ \eta(\ell_N (x), c_N) - 1]^q 
  \to 0
\end{equation}
The convergence is dominated since the assumption $\eta(\lambda, c)
\leq M \lambda$ implies that
$  |F (\ell_N (x), c_N)| \leq C(1 + |x|^q).$
Hence the convergence \eqref{eq:Fto0} along with \eqref{eq:S-dom} and
the dominated convergence theorem implies \eqref{eq:a-cge}.

For case (b),
\begin{equation*}
  N^{2q/3} \E (T_N) 
  = \int_{\delta_N}^\infty [ N^{2/3}(\ell_N(x) - \lambda(c_N))]_+^q
  S_{1 \tau} (x,x) dx. 
\end{equation*}
Observe that
\begin{equation*}
  N^{2/3}(\lambda(\gamma_N) - \lambda(c_N))
   \sim N^{2/3} \lambda'(\gamma) (\gamma_N - c_N)
   \sim - s \tau(\gamma),
\end{equation*}
and so from \eqref{eq:muNbd} and \eqref{eq:asyrel}, we have
\begin{equation}
  \label{eq:dom-b}
  N^{2/3}(\ell_N(x) - \lambda(c_N)) 
   = O(N^{-1/3}) + N^{2/3}(\lambda(\gamma_N) - \lambda(c_N)) + N^{2/3}
   n^{-1} \sigma_N x 
   \sim \tau(\gamma) (x - s).
\end{equation}
In addition, from \eqref{eq:dom-b}, we have
\begin{equation*}
  N^{2/3} |\ell_N(x) - \lambda(c_N) | \leq M (1 + |x|),
\end{equation*}
so that the convergence is dominated and \eqref{eq:cge-b} is proven.

For case (c), we have only to evaluate
\begin{equation*}
  c_0 = \int_0^\infty K_1(x,x) dx
     = \int_0^\infty K_A(x,x) dx  + \tfrac{1}{4} \int_0^\infty G'(x)
     dx 
     = I_1 + I_2,
\end{equation*}
where $I_1$ was evaluated in the previous section and 
$G(x) = [1 - \int_x^\infty \Ai(z) dz ]^2$. 
Since $\int_0^\infty \Ai(z) dz = 1/3$, from
\citet[9.10.11]{Olver:2010:NHMF}, we obtain
\begin{equation*}
  4 I_2 = G(\infty) - G(0) = 1 - (2/3)^2 = 5/9,
\end{equation*}
with the result
\begin{equation*}
  c_0 = \frac{1}{6 \sqrt{3} \pi} + \frac{5}{36} \approx 0.031 + 0.139 
      = 0.16952.
\end{equation*}

\section{Application to covariance estimation}
\label{sec:appl-covar-estim}

We indicate how Theorem \ref{prop:wish-case} is applied to
covariance estimation in the spiked model studied in \citet{dgj15}. 
Consider a sequence of statistical problems indexed by 
dimension $p$ and sample size $n$.
In the $n$th problem $\check X \sim {\rm N}_p(0,\Sigma)$ where $p = p_n$
sastisfies $p_n/n \to \gamma  \in (0,1]$ and the population covariance
matrix $\Sigma = \Sigma_p$ has fixed ordered eigenvalues $\ell_1 \geq \ldots
\geq \ell_r > 1$ for all $n$, and then $\ell_{r+1} = \ldots = \ell_{p_n} =
1$. 

Suppose that the sample covariance matrix $\check S = \check
S_{n,p_n}$ has eigenvalues
$\check \lambda_1 \geq \ldots \geq \check \lambda_p$ and corresponding
eigenvectors $v_1, \ldots, v_p$.
Consider shrinkage estimators of the form
\begin{equation}
  \label{eq:SigHat}
  \hat \Sigma_\eta = \sum_{j=1}^p \eta(\check \lambda_j, c_p) v_j v_j^\top,
\end{equation}
where $\eta(\lambda,c)$ is a \textit{continuous bulk shrinker}, 
that is, satisfies the conditions (a) of Theorem \ref{prop:wish-case}.
Without loss of generality, as explained in the reference cited, we
may also assume 
that $\lambda \to \eta(\lambda,c)$ is non-decreasing.
In the spiked model, the typical choice for $c_p$ in practice would be
to set $c_p = p/n$, and we adopt this choice below.

It is useful to analyse an ``oracle'' or ``rank-aware''
variant of \eqref{eq:SigHat} which
takes advantage of the assumed structure of $\Sigma_p$, especially the
fixed rank $r$ of $\Sigma_p - I$:
\begin{equation*}
  \hat \Sigma_{\eta,r} = \sum_{j=1}^r \eta(\check \lambda_j, c_p) v_j
  v_j^\top + \sum_{j=r+1}^p  v_j v_j^\top.
\end{equation*}

The error in estimation of $\Sigma$ using $\hat \Sigma$ is measured by
a loss function $L_p(\Sigma, \hat \Sigma)$.
One seeks conditions under which the losses $L_p(\Sigma, \hat
\Sigma_\eta)$ and $L_p(\Sigma, \hat \Sigma_{\eta,r})$ are
asymptotically equivalent.
They consider a large class of loss functions which satisfy a
Lipschitz condition which implies that, for some $q$,
\begin{equation*}
  | L_p(\Sigma, \hat \Sigma_{\eta}) - L_p(\Sigma, \hat
  \Sigma_{\eta,r})| 
   \leq C(\ell_1, \eta(\check \lambda_1))
         \sum_{j=r+1}^p [\eta(\check \lambda_j, c_p) - 1]^q.
\end{equation*}
Suppose now that  $\Pi:\mathbb{R}^{p}\to\mathbb{R}^{p-r}$ 
is a projection on the span of the $p-r$ unit eigenvectors of
$\Sigma$. 
Let $X = \Pi \check X$ and 
let $\lambda_{1} \geq \dots \geq \lambda_{p-r}$ denote 
the eigenvalues of 
$n^{-1} X X^\top$. By the Cauchy interlacing Theorem 
(e.g. \cite[p. 59]{bhat97}), we have 
\begin{equation}
  \label{eq:interlace} 
\check \lambda_{j}\leq \lambda_{j-r} \qquad \text{for} \  r+1\leq j\leq p,
\end{equation}
where the $(\lambda_{i})_{i=1}^{p-r}$ are the eigenvalues of a white
Wishart matrix  
$W_{p -r}(n,I)$. 
From the monotonicity of $\eta$,
\begin{equation}
  \label{eq:bd}
  \sum_{j=r+1}^p [\eta(\check \lambda_j, c_p) - 1]^q
   \leq 
  \sum_{i=1}^{p-r} [\eta(\lambda_i, c_p) - 1]^q.
\end{equation}
Now apply part (a) of Theorem \ref{prop:wish-case} with the identifications
\begin{equation*}
  N \leftarrow p-r, \qquad \qquad c_N \leftarrow c_p.
\end{equation*}
Clearly $\gamma_N = N/n \to \gamma$ and
\begin{equation*}
  c_N - \gamma_N = \frac{N+r}{n} - \frac{N}{n} = O(N^{-2/3}),
\end{equation*}
since $r$ is fixed.
We conclude that the right side of \eqref{eq:bd} and hence
$| L_p(\Sigma, \hat \Sigma_{\eta}) - L_p(\Sigma, \hat
  \Sigma_{\eta,r})|$ converge to $0$ in $L_1$ and in probability.

Part (c) of Theorem \ref{prop:wish-case} helps to give an example
where the losses $L_p(\Sigma, \hat
\Sigma_\eta)$ and $L_p(\Sigma, \hat \Sigma_{\eta,r})$ are \textit{not}
asymptotically equivalent.
Indeed, let $L_p(\Sigma, \hat{\Sigma}_\eta) = 
\| \hat \Sigma^{-1}_\eta - \Sigma^{-1} \| $, with $\| \cdot \|$
denoting matrix operator norm. 
Here the optimal shrinkage rule $\eta = \eta^*(\lambda,c)$ is
discontinuous at the upper bulk edge $\lambda(c) = (1 + \sqrt c)^2$:
\begin{alignat*}{2}
   \eta^*(\lambda,c) & = 1 \qquad & \text{for } \lambda & \leq
   \lambda(c) \\
   \eta^*(\lambda,c) & \to 1 + \sqrt c \qquad & \text{for } \lambda 
        & \downarrow  \lambda(c).
\end{alignat*}
Proposition 3 of \citet{dgj15} shows that 
\begin{equation}
  \label{eq:O-2}
    \| \hat \Sigma^{-1}_\eta - \Sigma^{-1} \|
  - \| \hat \Sigma^{-1}_{\eta,r} - \Sigma^{-1} \|
  \stackrel{\mathcal{D}}{\to} W,
\end{equation}
where $W$ has a two point distribution $(1-\pi) \delta_0 + \pi
\delta_w$ with non-zero probability 
\newline $\pi = \Pr ( TW_1 > 0)$ at
location
$w = f(\ell_+) - f(\ell_r)$, where 
$\ell_+ = 1 + \sqrt c$ and the function
\begin{equation*}
  f(\ell) =  \left[ \frac{c (\ell-1)}{\ell(\ell -1+\gamma)}
  \right]^{1/2} 
\end{equation*}
is strictly decreasing for $\ell \geq \ell_+$.

Part (c) of Theorem \ref{prop:wish-case}, along with interlacing
inequality \eqref{eq:interlace}, is used in the proof to establish
that $N_n = \# \{ i \geq r+1 : \lambda_{in} > \lambda_+(c_n) \} $, the
number of noise eigenvalues exiting the bulk, is bounded in probability.

\section{Final Remarks}
\label{sec:final-remarks}


It is apparent that the same methods will show that the value of $c_0$
for the Gaussian Orthogonal Ensemble
 will be the same as for the real Wishart 
(Laguerre Orthogonal Ensemble), and similarly
that the value of $c_0$ for the white complex Wishart 
(Laguerre Unitary Ensemble) will agree
with that for GUE. 

Some natural questions are left for further work. First, the
evaluation of $c_0$ for values of $\beta$ other than $1$ and $2$, and
secondly universality, i.e. that the limiting constants
do not require the assumption of Gaussian matrix entries.

\bigskip

\medskip 
Finally, this article appears in a special issue dedicated to
the memory of Peter Hall.  Hall's many contributions to high
dimensional data have been reviewed by \citet{samw16}. However, it
seems that Peter did not publish specifically on problems connected with
the application of random matrix theory to statistics --- the exception
that proves the rule of his extraordinary breadth and depth of
interests.  Nevertheless the present author's work on this specific topic, as
well as on many others, has been notably advanced by Peter's support
--- academic, collegial and financial -- in promoting research visits
to Australia and contact with specialists there in random matrix
theory, particularly at the University of Melbourne, Peter's academic
home since 2006.



\bibliographystyle{plainnat}

\bibliography{proof}
\end{document}